\documentclass[12pt]{article}
\usepackage[utf8]{inputenc}
\usepackage{amsmath}
\usepackage[english]{babel}
\usepackage{url}

\usepackage[a4paper,top=3cm,bottom=2cm,left=3cm,right=3cm,marginparwidth=1.75cm]{geometry}

\usepackage{amsmath}
\usepackage{graphicx}
\usepackage[colorinlistoftodos]{todonotes}
\usepackage[colorlinks=true, allcolors=blue]{hyperref}
\usepackage{amssymb}
\usepackage{amssymb}
\usepackage{amsthm}
\usepackage{authblk}
\usepackage{amsmath}
\usepackage{amsfonts}
\usepackage{amssymb}
\usepackage{amscd}
\usepackage{latexsym}
\usepackage{breqn}
\usepackage{cite}
\usepackage{mathtools}
\mathtoolsset{showonlyrefs}

\theoremstyle{plain}
\newtheorem{thm}{Theorem}[section]
\newtheorem{lem}[thm]{Lemma}
\newtheorem{prop}[thm]{Proposition}
\newtheorem{cor}[thm]{Corollary}

\theoremstyle{definition}
\newtheorem{conj}{Conjecture}[section]
\newtheorem{question}{Question}[section]

\theoremstyle{remark}
\newtheorem*{rem}{Remark}

\title{Symmetric matrix representations of truncated Toeplitz operators on finite dimensional spaces}
\author{Ryan O'Loughlin \\ E-mail address: \href{mailto:R.OLoughlin@leeds.ac.uk}{R.OLoughlin@leeds.ac.uk}}
\affil{School of Mathematics, University of Leeds, Leeds, LS2 9JT, U.K.}
\date{}
\begin{document}

\maketitle
\begin{abstract}
In this paper, we answer an open conjecture concerning complex symmetric matrices and truncated Toeplitz operators. We study matrix representations of truncated Toeplitz operators with respect to orthonormal bases which are invariant under a canonical conjugation map. In particular, we determine necessary and sufficient conditions for when a symmetric matrix is the matrix representation of a truncated Toeplitz operator with respect to a given conjugation invariant orthonormal basis. We specialise our result to the case when the conjugation invariant orthonormal basis is a modified Clark basis. With this specialisation, we answer an open conjecture in the negative, and show not every unitary equivalence between a complex symmetric matrix and a truncated Toeplitz operator arises from modified Clark basis representations. We pose a new refined conjecture for how to realise a model theory for symmetric matrices through the use of truncated Toeplitz operators, and we show this conjecture is equivalent to a specified system of polynomial equations being satisfied.

\end{abstract}

\section{Introduction}
Since the seminal work of Sarason in 2007 \cite{sarason2007algebraic}, there has been a surge of interest in truncated Toeplitz operators (which we will abbreviate to TTOs) \cite{baranov2010bounded, baranov2010symbols, symbolsofcompact, garciaUETTOanalyticsymbols, bessenovfredholm, TTOUEsimilarity, MR3398735, UETTOstrouse, oloughlin2020nearly, mythesis}. In parallel to this, over the past decade there has also been a growing body of research devoted to complex symmetric operators \cite{CSO1}\cite{CSO2} \cite{CSO3} (definitions of complex symmetric operators and TTOs given in Section \ref{prelims}). Complex symmetric operators are a deceptively wide class of operators; normal, Volterra and rank-one operators are all examples of complex symmetric operators. Although complex symmetric operators have wide reaching applications in numerous fields (we refer the reader to \cite{garciaapp} and references thereafter for a detailed discussion of these applications), one noteworthy domain where complex symmetric operators appear naturally is non-Hermitian Quantum Physics. Non-Hermitian Quantum Physics does not require that the Hamiltonian be a self adjoint operator, but that the Hamiltonian is a symmetric operator which commutes with another so called $\mathcal{P T}$ operator. We refer the reader to \cite{nonhermitianphysics, garciaapp} for a introductory background on this topic, which emphasises how complex symmetric operators appear in this field. 



There is an emerging body of evidence to suggest that truncated Toeplitz operators serve as a model operator for complex symmetric operators. The following is an open question in \cite{garciaopen}.

\begin{question}\label{Q}
Just as multiplication operators $M_z : L^2(X, \mu ) \to L^2(X, \mu)$ play
a fundamental role in decomposing normal operators, can one develop a comparable model theory for complex symmetric operators? There is some indication that truncated Toeplitz operators may play an important role in the resolution of this problem.
\end{question}

\noindent
Decomposing normal operators (i.e. the Spectral Theorem) is a famous cornerstone result of Operator Theory and a positive answer to Question \ref{Q} would be a significant generalisation of the Spectral Theorem to non-normal operators. However as this field of study is still in its infancy, the above question is not yet fully understood even in the finite dimensional case. In finite dimensions, complex symmetric operators are the operators which are unitarily equivalent to a symmetric matrix. In this regard, the following question from \cite{recentprogressGarcia} is a finite dimensional specialisation of Question \ref{Q}.

\begin{question}\label{Q2}
Is every symmetric matrix unitarily equivalent to a direct sum of TTOs?
\end{question}
\noindent The question above is known to be true for $2$-by-$2$ and $3$-by-$3$ matrices (\cite{recentprogressGarcia} and \cite{garciaUETTOanalyticsymbols} respectively), rank one matrices \cite{garciaUETTOanalyticsymbols} and in several other cases.

When one wants to obtain a unitary equivalence between a symmetric matrix and a TTO, the only known constructive way to do this is to compute the matrix representation of the TTO with respect to a conjugation-invariant orthonormal basis (or $C_{\theta}$-real basis for short). The most notable class of $C_{\theta}$-real bases are the modified Clark bases (see Section 8 in \cite{CCOgarcia}). In this vein, the following question was posed in Section 7 of the article \cite{garciaUETTOanalyticsymbols}, and again in Section 9 of the book \cite{recentprogressGarcia}.


\begin{question}\label{gotoQ1}
Suppose that $M$ is a complex symmetric matrix. If $M$ is unitarily equivalent to a TTO, does there exist an inner function $u$, and a modified Clark basis for the model space ${K}_{u}$ such that $M$ is the matrix representation of a TTO on $K_u$ with respect to this basis? In other words, do all such unitary equivalences between complex symmetric matrices and TTOs arise from modified Clark basis representations?
\end{question}

\noindent Regarding a model theory for symmetric operators, Questions \ref{Q} and \ref{Q2} may be viewed as an existence result, whereas Question \ref{gotoQ1} is a constructive question which asks how to realise such a model theory.

The purpose of this paper is to study matrix representations of TTOs with respect to $C_{\theta}$-real bases in order to
answer Question \ref{gotoQ1} in the negative. Then we pose and study a new refined conjecture which hopes to model symmetric matrices through the use of TTOs.

In Section \ref{prelims} we provide the necessary preliminary background in order to study the matrix representations of TTOs with respect to $C_{\theta}$-real bases. In particular, we include a brief discussion on the construction of the modified Clark bases. 

We split Section \ref{sec3} into two subsections. In the first subsection, we show that a given symmetric matrix is the matrix representation of a TTO on $K_{\theta}$ with respect to a given $C_{\theta}$-real basis if and only if a corresponding matrix has non-zero kernel. In the second subsection we specialise the results of the first subsection, and show a given 3-by-3 symmetric matrix is the matrix representation of a TTO on $K_{\theta}$ with respect to a given modified Clark basis for $K_{\theta}$ if and only if the off diagonal entries satisfy a particular linear relation. With this condition, we then answer Question \ref{gotoQ1} in the negative.

In Section \ref{sec4} we show that every matrix representation of a TTO with respect to some $C_{\theta}$-real basis is orthogonally equivalent to the matrix representation of the same TTO with respect to a modified Clark basis. With this realisation we can then show that a given 3-by-3 symmetric matrix is the matrix is the matrix representation of a TTO on a model space $K_{\theta}$ with respect to a $C_{\theta}$-real basis for $K_{\theta}$ if and only if a specified system of polynomial equations is satisfied with a real solution.

\section{Preliminaries}\label{prelims}
We denote the open unit disc in the complex plane by $\mathbb{D}$ and the unit circle by $\mathbb{T}$. The Hardy space, $H^2$, is the space of all analytic functions $f(z) = \sum_{n=0}^{\infty}a_n z^n $ on $\mathbb{D}$ such that 
$$
\|f\|=\left(\sum_{n=0}^{\infty}\left\|a_n\right\|^2\right)^{\frac{1}{2}}<\infty .
$$
For each $\zeta \in \mathbb{T}$ and $f \in H^2$, we can define $f(\zeta) := \lim_{r \to 1} f(r \zeta)$ to be the radial limit of $f$ at $\zeta$. It can be shown that $f(\zeta)$ exists for almost every $\zeta \in \mathbb{T}$ and $f(\zeta) \in L^2 ( \mathbb{T})$, so with this correspondence we view $H^2 \subseteq L^2(\mathbb{T})$. We refer the reader to \cite{duren1970theory, nikolski2002operators} for a detailed background on the Hardy space. Throughout we let $\theta$ be an inner function (i.e. a function in $H^2$ which is unimodular on $\mathbb{T}$). A famous theorem of Beurling characterises all the forward shift invariant subspaces of $H^2$. It states that all nontrivial closed forward shift invariant subspaces are of the form $\theta H^2$, where $\theta$ is some inner function. The backward shift is the adjoint of the forward shift on $H^2$ and is given by the map $f(z) \mapsto \frac{f - f(0)}{z}$. It follows from Beurling's Theorem that all nontrivial closed backward shift invariant subspaces of $H^2$ are of the form $K_\theta: = H^2 \ominus \theta H^2 =  H^2 \cap \theta \overline{zH^2}$, where the multiplication $\theta \overline{zH^2}$ is understood as functions on the unit circle. The space $K_\theta$ is referred to as a \textit{model space}, and we refer the reader to \cite{cima2000backward} for a detailed background on model spaces.

Model spaces are reproducing kernel Hilbert spaces, where the reproducing kernel at $\lambda \in \mathbb{D}$ is given by \begin{equation}\label{RK}
    k_{\lambda}(z) =\frac{1-\overline{\theta(\lambda)} \theta(z)}{1-\bar{\lambda} z},
\end{equation} (see Section 5.5 of \cite{modelspacesgarcia}). A model space $K_\theta$ is finite dimensional if and only if $\theta$ is a finite Blaschke product (Proposition 5.19 in \cite{modelspacesgarcia}), and in this case the order of $\theta$ (i.e. the number of zeros of $\theta$ in the disc) is equal to the dimension of $K_{\theta}$. In the case when $\theta$ is a finite Blaschke product, for every $\zeta \in \mathbb{T}$ it can be seen that $\theta$ and $\theta' $ have a nontangential limit at $\zeta$ and $| \theta( \zeta) | =1$. By Theorem 2 in \cite{recentprogressGarcia} this is equivalent to every $f \in K_{\theta}$ having a nontangential limit at every $\zeta \in \mathbb{T}$, and thus every $f\in K_{\theta}$ can be defined everywhere on $\mathbb{T}$. As a result of this, when $\theta$ is a finite Blaschke product we can consider the reproducing kernel at $\zeta \in \mathbb{T}$, which is given by $$k_{\zeta}(z) =\frac{1-\overline{\theta(\zeta)} \theta(z)}{1-\bar{\zeta} z} \in K_{\theta}.$$ For vectors $x, y \in K_{\theta}$ we use the notation $x \otimes y : K_{\theta} \to K_{\theta}$ to mean the map $f \mapsto \langle f, y \rangle x$

If $\mathcal{H}$ is a Hilbert space, then we say that $C$ is a conjugation operator on $\mathcal{H}$ if the following conditions hold:
\newline 
(a) $C$ is antilinear:
$$
C\left(a_{1} f_{1}+a_{2} f_{2}\right)=\overline{a_{1}} C f_{1}+\overline{a_{2}} C f_{2}
$$
for all $a_{1}, a_{2} \in \mathbb{C}$ and $f_{1}, f_{2} \in \mathcal{H}$.\newline
(b) $C$ is isometric:
\begin{equation}\label{b}
\langle C f, C g\rangle=\langle g, f\rangle
\end{equation}
for all $f, g \in \mathcal{H}$.
\newline
(c) $C$ is involutive: \begin{equation}\label{c}
C^{2}=I .
\end{equation}
\newline An operator $F$ on a Hilbert space $\mathcal{H}$ is said to be \textit{$C$-symmetric} if $CFC = F^*$, and an operator is \textit{complex symmetric} if it is $C$-symmetric with respect to some conjugation map $C$.

The space $K_{\theta}$ carries a canonical conjugation operator given by $C_{\theta}f = \theta \overline{zf}$. This model space conjugation operator is discussed at length in \cite{CCOgarcia}. Furnishing an orthonormal basis for $K_{\theta}$ which is invariant under $C_{\theta}$ turns out to be a non-trivial task. Such a basis is called a \textit{$C_{\theta}$-real basis}, and Lemma 2.6 in \cite{CCOgarcia} shows every model space $K_{\theta}$ admits a $C_{\theta}$-real basis.

Although a detailed explanation of how to construct a $C_{\theta}$-real basis for a model space is given by Section 8 in \cite{CCOgarcia}, we outline the approach here. In the case when the inner function $\theta$ is a finite Blaschke product, a $C_{\theta}$-real basis is given by the unit eigenvectors of a generalised Clark operator $U_{t, \alpha} : K_{\theta} \to K_{\theta}$,
\begin{equation}\label{clarkoperator}
    U_{t, \alpha}=S_{t}+(\alpha+\theta(t))\left(k_{t} \otimes C_{\theta} k_{t}\right),
\end{equation}
where $t \in \mathbb{D}$, $ \alpha \in \mathbb{T}$ and $S_{t} f=P_{\theta}\left(\frac{z-t}{1-\bar{t} z} f\right)$ for $P_{\theta} : L^2 \to K_{\theta}$, the orthogonal projection. We refer to a $C_{\theta}$-real basis of this form as a \textit{modified Clark basis}. Explicitly the unit eigenvectors of $U_{t, \alpha}$ are given by the formula
\begin{equation}\label{explicitcb}
    \textbf{cb}_i(z)=\left(\overline{\eta_{i}} \frac{\alpha+\theta(t)}{1+\overline{\theta(t)} \alpha}\right)^{\frac{1}{2}} \frac{k_{\eta_{i}}}{\|k_{\eta_{i}}\|},
\end{equation}
where $\eta_i$ are the points such that 
\begin{equation}\label{above}
    \theta\left(\eta_{i}\right)=\frac{\alpha+\theta(t)}{1+\overline{\theta(t)} \alpha},
\end{equation} and where we take the convention that the above square root is defined by taking $\delta_1 = \arg (\overline{\eta_i}) \in [0, 2 \pi)$, $\delta_2 = \arg \left( \frac{\alpha+\theta(t)}{1+\overline{\theta(t)} \alpha} \right ) \in [0, 2 \pi)$ and setting $\arg \left(\left( \overline{\eta_i} \frac{\alpha+\theta(t)}{1+\overline{\theta(t)} \alpha}\right)^{\frac{1}{2}} \right) = \frac{\delta_1 + \delta_2}{2}$. Observe that each $t \in \mathbb{D}$, $ \alpha \in \mathbb{T}$ will give a different corresponding modified Clark basis. We will use the description of the modified Clark basis given by \eqref{explicitcb} throughout.

If $\theta$ is order $n$, then Theorem 3.4.10 in \cite{FBPandconnections} ensures that there are precisely $n$ distinct $\eta_i$ satisfying \eqref{above} and each $\eta_i$ lies on $\mathbb{T}$. We remark that either choice of the square root will give a $C_{\theta}$-real basis, however our choice of square root will make our working in later sections easier to follow.

Let $\theta$ be a finite Blaschke product. For $\phi \in L^{\infty}(\mathbb{T})$, we define the truncated Toeplitz operator (which we have abbreviated to TTO), $A_{\phi}:  K_{\theta} \to K_{\theta}$, by 
$$
A_{\phi}(f) = P_{\theta} ( \phi f ).
$$
For the Blaschke product $\theta(z) =  z^n$, we have  $K_{z^n} = \operatorname{span} 1,z,z^2, ..., z^{n-1}$ and every TTO on $K_{z^n}$ is a Toeplitz matrix, so TTOs may be viewed as generalisations of Toeplitz matrices. Although the TTO implicitly depends on our choice of $\theta$ we will omit this from our notation and we will often suppress the $\phi$ subscript and just write $A$ instead of $A_{\phi}$.

Throughout we denote $\mathcal{T}_{\theta}$ to be the space of all TTOs on the model space $K_{\theta}$. For $A \in \mathcal{T}_{\theta}$ and a $C_{\theta}$-real basis for $K_{\theta}$ given by $v_j$ where $j=1,..., \dim K_{\theta}$, we denote $[A]_{v_j}$ to be the matrix representation of $A$ with respect to $v_j$. Lemma 2.7 in \cite{CCOgarcia} shows that the matrix representation of a $C$-symmetric operator with respect to a $C$-real basis is a symmetric matrix and so as $A \in \mathcal{T}_{\theta}$ is $C$-symmetric with respect to the canonical conjugation $C_{\theta}$ (see Lemma 2.1 in \cite{sarason2007algebraic}), we must have that $[A]_{{v}_j}$ is a symmetric matrix. For this reason, when we are trying to realise which matrices are the matrix representation of a TTO with respect to a $C_{\theta}$-real basis we in fact only need to consider which symmetric matrices are the representation of a TTO with respect to a $C_{\theta}$-real basis.


\section{Matrix representations of TTOs with respect to $C_{\theta}$-real bases}\label{sec3}

\subsection{General $C_{\theta}$-real basis}\label{sec3.1}
The following is a generalisation of Theorem 7.1 (b) in \cite{sarason2007algebraic}.
\begin{lem}
Let $\theta$ be a Blaschke product of order 3, let $t_1, t_2, t_3$ be distinct points in $ \mathbb{T}$ and $\lambda_4, \lambda_5$ distinct points in $ \mathbb{D}$, then 
\begin{equation}\label{specified}
    k_{t_1} \otimes k_{t_1} \quad k_{t_2} \otimes k_{t_2} \quad k_{t_3} \otimes k_{t_3} \quad k_{\lambda_4} \otimes C_{\theta}k_{\lambda_4} \quad k_{\lambda_5} \otimes C_{\theta}k_{\lambda_5}
\end{equation}
is a basis for $\mathcal{T}_{\theta}$.
\end{lem}
\begin{proof}
Theorem 5.1 in \cite{sarason2007algebraic} shows that the specified functions are in indeed in $\mathcal{T}_{\theta}$ and Theorem 7.1(a) in \cite{sarason2007algebraic} shows the dimension of $\mathcal{T}_{\theta}$ is 5, so it suffices to show that the functions given by \eqref{specified} are linearly independent.

We first show that any 3 distinct reproducing kernels in $K_{\theta}$ are linearly independent. The well known result, which can be found as Corollary 5.18 in \cite{modelspacesgarcia}, shows that if the zeroes of $\theta$ are given by $z_n$ (allowing for repeated zeroes) then 
$$
K_{\theta} = \left\{ \frac{a_0 +a_1 z + a_2 z^2}{(1-\overline{z_1}z)(1-\overline{z_2}z) (1-\overline{z_3}z)} \, : \, a_1, a_2 , a_3 \in \mathbb{C}  \right\}.
$$
Thus for any three distinct points $x_1 , x_2 , x_3$ in $\mathbb{D} \cup \mathbb{T}$, we can define a function 
$$
f(z) := \frac{(z-x_2)(z-x_3)}{(1-\overline{z_1}z)(1-\overline{z_2}z) (1-\overline{z_3}z)} \in K_{\theta},
$$
such that $f(x_2) =f(x_3) = 0$, but $f(x_1) \neq 0$. So if for any three distinct points $x_1 , x_2 , x_3$ in $\mathbb{D} \cup \mathbb{T}$ we have 
$$
\sum_{i=1}^3 \mu_i k_{x_i} = 0,
$$
for some constants $\mu_i$, then for the $f$ defined above we have 
$$
0 = \langle f , \sum_{i=1}^3 \mu_i k_{x_i} \rangle = \overline{\mu_1} f(x_1),
$$
which implies $\mu_1 =0$. We can repeat a similar argument to show $\mu_2, \mu_3$ are both 0.

So if we have complex constants $\alpha_1, ..., \alpha_5$ such that 
$$
\sum_{i=1}^3 \alpha_i (k_{t_i} \otimes k_{t_i}) + \sum_{i=4}^5 \alpha_i (k_{\lambda_i} \otimes C_{\theta} k_{\lambda_i}) = 0,
$$
then for $f \in K_{\theta}$ such that $f(t_2) = f(t_3) = 0$  and $f(t_1) \neq 0$, we have 
\begin{equation}\label{Cthetaf}
     0=\left(\sum_{i=1}^3 \alpha_i (k_{t_i} \otimes k_{t_i}) + \sum_{i=4}^5 (\alpha_i k_{\lambda_i} \otimes C_{\theta} k_{\lambda_i}) \right) (f) = {\alpha_1} f(t_1) k_{t_1} + \sum_{i=4}^5 {\alpha_i} k_{\lambda_i} \langle f, C_{\theta} k_{\lambda_i} \rangle .
\end{equation}
Now linear independence of $k_{t_1}, k_{\lambda_4}, k_{\lambda_5}$ guarantees $\alpha_1 =0$.

A similar argument shows $\alpha_2, \alpha_3$ are $0$ and from here the linear independence of $k_{\lambda_4} \otimes C_{\theta} k_{\lambda_4}, k_{\lambda_5} \otimes C_{\theta} k_{\lambda_5}$ ,which is shown by Theorem 7.1 in \cite{sarason2007algebraic}, ensures $\alpha_4, \alpha_5$ are 0. (Alternatively one can show $\alpha_4, \alpha_5$ are $0$, by noticing that for $f \in K_{\theta}$ such that $f(\lambda_4) = 0$, $f(\lambda_5) \neq 0$ we have $ 0 =  \left(\sum_{i=4}^5 (\alpha_i k_{\lambda_i} \otimes C_{\theta} k_{\lambda_i})\right) (C_{\theta}(f))  = \alpha_5 \overline{f( \lambda_5)} k_{\lambda_5}$, and so $\alpha_5 = 0$.) 

\end{proof}

\begin{thm}\label{detthm}
Let $\theta$ be a Blaschke product of order 3, let ${v}_1, {v}_2, {v}_3$ be a $C_{\theta}$-real basis for $K_{\theta}$, let $t_1, t_2, t_3$ be distinct points on $\mathbb{T}$ and $\lambda_4, \lambda_5$ be distinct points in $\mathbb{D}$. Then a symmetric matrix $$
S = \begin{pmatrix}
s_1 & s_4 & s_5 \\
s_4 & s_2 & s_6 \\
s_5 & s_6 & s_3
\end{pmatrix},
$$
is a matrix representation of a TTO on $K_{\theta}$ with respect to $v_j$ if and only if
$$
\det [c_1 , c_2, c_3, c_4, c_5, \tilde{S}] =0,
$$
where \begin{equation}\label{c_i}
    c_i = \begin{pmatrix}
v_1(t_i) \overline{v_1(t_i)} \\
v_2(t_i) \overline{v_2(t_i)} \\
v_3(t_i) \overline{v_3(t_i)} \\
v_1(t_i) \overline{v_2(t_i)} \\
v_1(t_i) \overline{v_3(t_i)} \\
v_2(t_i) \overline{v_3(t_i)} 
\end{pmatrix} \text{ for } i=1,2,3 , \quad  c_i = \begin{pmatrix}
\overline{v_1(\lambda_i)v_1(\lambda_i)} \\
 \overline{v_2(\lambda_i)v_2(\lambda_i)} \\
 \overline{v_3(\lambda_i)v_3(\lambda_i)} \\
 \overline{v_1(\lambda_i)v_2(\lambda_i)} \\
 \overline{v_1(\lambda_i)v_3(\lambda_i)} \\
 \overline{v_2(\lambda_i)v_3(\lambda_i)} 
\end{pmatrix} \text{ for } i=4,5,
\end{equation} and \begin{equation}\label{S}
    \tilde{S} = \begin{pmatrix}
s_1 \\
s_2 \\
s_3 \\
s_4 \\
s_5 \\
s_6 
\end{pmatrix}.
\end{equation}

\end{thm}
\begin{rem}
This theorem may easily be altered so that $t_1, t_2, t_3,  \lambda_4, \lambda_5$ may take any values in $ \mathbb{D} \, \cup \, \mathbb{T}$. The reason we phrase the theorem so that $t_1, t_2, t_3 \in \mathbb{T}$ and $\lambda_4, \lambda_5 \in \mathbb{D}$ is because this will allow us to easily specialise the result to the case when $v_j$ is a modified Clark basis in the next subsection.
\end{rem}

\begin{proof}
We first compute $[k_{t_i} \otimes k_{t_i}]_{v_j}$ for $i=1,2,3$ and $[k_{\lambda_i} \otimes C_{\theta}k_{\lambda_i}]_{v_j}$ for $i=4,5$. For $i=1,2,3$ we have
$$
k_{t_i} = \sum_{j=1}^3 \langle k_{t_i}, v_j \rangle v_j = \sum_{j=1}^3 \overline{v_j(t_i)} v_j,
$$
and so 
$$
[k_{t_i} \otimes k_{t_i}]_{v_j} = \begin{pmatrix}
v_1(t_i)\begin{pmatrix}
\overline{v_1(t_i)} \\
\overline{v_2(t_i)} \\
\overline{v_3(t_i)} \\
\end{pmatrix} & \, v_2(t_i) \begin{pmatrix}
\overline{v_1(t_i)} \\
\overline{v_2(t_i)} \\
\overline{v_3(t_i)} \\
\end{pmatrix} & \, v_3(t_i) \begin{pmatrix}
\overline{v_1(t_i)} \\
\overline{v_2(t_i)} \\
\overline{v_3(t_i)} \\
\end{pmatrix} 
\end{pmatrix}.
$$
Similarly for $i=4,5$ as each $v_j$ is $C_{\theta}$-invariant using properties \eqref{b} \eqref{c} we have $\langle v_j, C_{\theta}k_{\lambda_i} \rangle = \overline{v_j(\lambda_i)}$, and so
$$
[k_{\lambda_i} \otimes C_{\theta}k_{\lambda_i}]_{v_j}=  \begin{pmatrix}
\overline{v_1(\lambda_i)}\begin{pmatrix}
\overline{v_1(\lambda_i)} \\
\overline{v_2(\lambda_i)} \\
\overline{v_3(\lambda_i)} \\
\end{pmatrix} & \overline{v_2(\lambda_i)} \begin{pmatrix}
\overline{v_1(\lambda_i)} \\
\overline{v_2(\lambda_i)} \\
\overline{v_3(\lambda_i)} \\
\end{pmatrix} & \overline{v_3(\lambda_i)} \begin{pmatrix}
\overline{v_1(\lambda_i)} \\
\overline{v_2(\lambda_i)} \\
\overline{v_3(\lambda_i)} \\
\end{pmatrix} 
\end{pmatrix}.
$$
By the previous lemma we can say that $[k_{t_i} \otimes k_{t_i}]_{v_j}$ for $i=1,2,3$ $[k_{\lambda_i} \otimes C_{\theta}k_{\lambda_i}]_{v_j}$ for $i=4,5$ is a basis for the space of matrix representations of $\mathcal{T}_{\theta}$ with respect to $v_j$ (i.e. a basis for $ [ \mathcal{T}_{\theta} ]_{ v_j} = \{ [A]_{v_j} \text{ such that } A \in \mathcal{T}_{\theta} \}$). 

So for a matrix $$S = \begin{pmatrix}
s_1 & s_4 & s_5 \\
s_4 & s_2 & s_6 \\
s_5 & s_6 & s_3
\end{pmatrix},
$$
there exists an $A \in \mathcal{T}_{\theta}$ such that $[A]_{v_j} =S$ if and only if there exists $\mu_1, ... , \mu_5 \in \mathbb{C}$ such that 
\begin{equation}\label{symmetry}
    \sum_{i=1}^3 \mu_i \begin{pmatrix}
v_1(t_i)\begin{pmatrix}
\overline{v_1(t_i)} \\
\overline{v_2(t_i)} \\
\overline{v_3(t_i)} \\
\end{pmatrix} & v_2(t_i) \begin{pmatrix}
\overline{v_1(t_i)} \\
\overline{v_2(t_i)} \\
\overline{v_3(t_i)} \\
\end{pmatrix} & v_3(t_i) \begin{pmatrix}
\overline{v_1(t_i)} \\
\overline{v_2(t_i)} \\
\overline{v_3(t_i)} \\
\end{pmatrix} 
\end{pmatrix} 
\end{equation}
$$ + \newline  \sum_{i=4}^5 \mu_i \begin{pmatrix}
\overline{v_1(\lambda_i)}\begin{pmatrix}
\overline{v_1(\lambda_i)} \\
\overline{v_2(\lambda_i)} \\
\overline{v_3(\lambda_i)} \\
\end{pmatrix} & \overline{v_2(\lambda_i)} \begin{pmatrix}
\overline{v_1(\lambda_i)} \\
\overline{v_2(\lambda_i)} \\
\overline{v_3(\lambda_i)} \\
\end{pmatrix} & \overline{v_3(\lambda_i)} \begin{pmatrix}
\overline{v_1(\lambda_i)} \\
\overline{v_2(\lambda_i)} \\
\overline{v_3(\lambda_i)} \\
\end{pmatrix} 
\end{pmatrix} = S.
$$

As noted in the previous section, each summand in the above expression must be a symmetric matrix. So using the symmetry of the matrices, we see when considering whether there exists $\mu_1, ..., \mu_5$ such that \eqref{symmetry} holds, we in fact only need to consider whether there are $\mu_1, ..., \mu_5$ such that

\begin{equation}\label{symmetry2}
    \sum_{i=1}^3 \mu_i \begin{pmatrix}
v_1(t_i)\begin{pmatrix}
\overline{v_1(t_i)} \\
\overline{v_2(t_i)} \\
\overline{v_3(t_i)} \\
\end{pmatrix} & v_2(t_i) \begin{pmatrix}
0 \\
\overline{v_2(t_i)} \\
\overline{v_3(t_i)} \\
\end{pmatrix} & v_3(t_i) \begin{pmatrix}
0 \\
0 \\
\overline{v_3(t_i)} \\
\end{pmatrix} 
\end{pmatrix} 
\end{equation}
$$ + \newline  \sum_{i=4}^5 \mu_i \begin{pmatrix}
\overline{v_1(\lambda_i)}\begin{pmatrix}
\overline{v_1(\lambda_i)} \\
\overline{v_2(\lambda_i)} \\
\overline{v_3(\lambda_i)} \\
\end{pmatrix} & \overline{v_2(\lambda_i)} \begin{pmatrix}
0 \\
\overline{v_2(\lambda_i)} \\
\overline{v_3(\lambda_i)} \\
\end{pmatrix} & \overline{v_3(\lambda_i)} \begin{pmatrix}
0 \\
0 \\
\overline{v_3(\lambda_i)} \\
\end{pmatrix} 
\end{pmatrix} = \begin{pmatrix}
s_1 & 0 & 0 \\
s_4 & s_2 & 0 \\
s_5 & s_6 & s_3
\end{pmatrix}.
$$
Rewriting this in terms of column vectors, we see that this is actually equivalent to the existence of $\mu_1, ..., \mu_5 \in \mathbb{C}$ such that $\sum_{i=1}^5 \mu_i c_i = \tilde{S}$, or equivalently (as $c_1 , c_2, c_3, c_4, c_5$ are linearly independent)
$$
\det [c_1 , c_2, c_3, c_4, c_5, \tilde{S}] =0. \qedhere
$$
\end{proof}
The above theorem generalises to the case when $\theta $ is a Blaschke product of order $n \in \mathbb{N}$. To do this we define the vectorialisation map. Let $S_n(\mathbb{C})$ denote the $n$-by-$n$ symmetric matrices with complex entries. The vectorialisation map, denoted $V :  S_{n} (\mathbb{C}) \to \mathbb{C}^{\frac{n(n+1)}{2}}$, is defined by $M = (a_{ij}) \mapsto V(M)$ where
$$
V(M) = \begin{pmatrix}
a_{11} \\
a_{12} \\
\vdots \\
a_{1n} \\
a_{22} \\
a_{23} \\
\vdots \\
a_{2n} \\
a_{33} \\
\vdots \\
a_{3n} \\
\vdots \\
\vdots \\
a_{(n-1)(n-1)} \\
a_{(n-1) n} \\
a_{nn}
\end{pmatrix}.
$$
It is readily verified that $V$ is an injective $\mathbb{C}$-linear map, and so as $\dim S_n(\mathbb{C}) =  \dim \mathbb{C}^{\frac{n(n+1)}{2}}$ it also follows that $V$ is surjective and hence a linear isomorphism.
\begin{thm}
Let $\theta$ be a Blaschke product of order $n$, let ${v}_1, ..., {v}_n$ be a $C_{\theta}$-real basis for $K_{\theta}$ and let $\lambda_1, ... , \lambda_n$ be distinct points in $\mathbb{D}$. Then a symmetric matrix $$
S = (s_{ij})
$$
is a matrix representation of a TTO on $K_{\theta}$ with respect to $v_1, ..., v_n$ if and only if
$$
\ker [c_1 , c_2, ...  c_{2n-1}, V(S)] \neq \{ 0 \},
$$
where $c_l$ is the vectorialisation of the symmetric matrix with $i,j$'th entry given by $a_{ij}^l= \overline{v_i(\lambda_l) v_j(\lambda_l) }$.
\end{thm}
\noindent The following proof closely resembles the proof of the previous theorem, so we only outline the details of the proof.
\begin{proof}
Invoke Theorem 7.1 in \cite{sarason2007algebraic} to show that $[k_{\lambda_l} \otimes C_{\theta}k_{\lambda_l}]_{v_1, ... ,v_n}$ for $l=1, ..., 2n-1$ is a basis for $\mathcal{T}_{\theta}$. It can be computed that $[k_{\lambda_l} \otimes C_{\theta}k_{\lambda_l}]_{v_1, ... ,v_n}$ is the $n$-by-$n$ matrix with $i,j$'th entry given by 
$$
 a_{ij}^l= \overline{v_i(t_l) v_j(t_l) } = a_{ji}^l.
$$
Thus, there exists an $A \in \mathcal{T}_{\theta}$ such that $[A]_{v_1, ... ,v_n} = S$ if and only if there exists constants $\mu_1, ... , \mu_{2n-1} \in \mathbb{C}$ such that
$$
\sum_{l=1}^{2n-1} \mu_l [k_{\lambda_l} \otimes C_{\theta}k_{\lambda_l}]_{v_1, ... ,v_n} = S.
$$
Now applying the map $V$ and using our previous observation that $V$ is a linear isomorphism, we see this is equivalent to the existence of $\mu_1, ... , \mu_{2n-1} \in \mathbb{C}$ such that
$$
\sum_{l=1}^{2n-1} \mu_l c_l = V(S),
$$
which is equivalent to
$$
\ker [c_1 , c_2, ...  c_{2n-1}, V(S)] \neq \{ 0 \}.
$$
\end{proof}

\subsection{When the $C_{\theta}$-real basis is a modified Clark basis}\label{sec3.2}
In this subsection we specialise the results of the previous subsection to the case where the $C_{\theta}$-real basis is a modified Clark basis
\begin{lem}\label{newlem}
Let $\theta$ be a finite Blaschke product and let $\eta \in \mathbb{T}$. Then 
$$
\| k_{\eta} \|^2 = \frac{\theta'(\eta) \eta }{\theta(\eta)} = | \theta' (\eta) | .
$$
\end{lem}
\begin{proof}
$$
\frac{\theta(\eta)}{\eta} \| k_{\eta} \|^2 = \frac{\theta(\eta)}{\eta} \lim_{z \to \eta}\langle k_{\eta}, k_z \rangle = \frac{\theta(\eta)}{\eta} \lim_{z \to \eta} \frac{1 - \overline{\theta(\eta)}\theta(z)}{1 - \overline{\eta}z} = \lim_{z \to \eta} \frac{\theta(z) - \theta(\eta)}{z - \eta} = \theta'(\eta),
$$
where the penultimate equality holds because $| \theta(\eta)|^2 = 1 = |\eta|^2$. Thus $\| k_{\eta} \|^2 = \frac{\theta'(\eta) \eta }{\theta(\eta)} $, and by taking the modulus we deduce $\| k_{\eta} \|^2 = \frac{|\theta'(\eta)| |\eta| }{|\theta(\eta)|} = |\theta'(\eta)|.$
\end{proof}

\begin{thm}\label{clarkbasisrepn}
Let $\theta$ be a Blaschke product of order 3 and $\textbf{cb}_i$ be modified Clark basis for $K_{\theta}$. Then a symmetric matrix $$
S = \begin{pmatrix}
s_1 & s_4 & s_5 \\
s_4 & s_2 & s_6 \\
s_5 & s_6 & s_3
\end{pmatrix},
$$ is a matrix representation of a TTO on $K_{\theta}$ with respect to $\textbf{cb}_i$ if and only if
\begin{equation}\label{solntos6}
s_6 = \frac{1}{ \eta_3 - \eta_2} \left( s_4 \left(\frac{\overline{\eta_{1}}  ^{\frac{1}{2}}|\theta'(\eta_1)|^{\frac{1}{2}}}{\overline{\eta_{3}} ^{\frac{1}{2}}|\theta'(\eta_3)|^{\frac{1}{2}}}\right) ( \eta_1 - \eta_2 ) + s_5 \left(\frac{\overline{\eta_{1}}  ^{\frac{1}{2}}|\theta'(\eta_1)|^{\frac{1}{2}}}{\overline{\eta_{2}} ^{\frac{1}{2}}|\theta'(\eta_2)|^{\frac{1}{2}}}\right)(\eta_3 - \eta_1 ) \right),
\end{equation}
where the $\eta_i$ are given by \eqref{above} and where $\overline{\eta_{j}}  ^{\frac{1}{2}}$ is defined such that if $\arg( \overline{\eta_j}) := \gamma \in [0, 2 \pi)$, then $\arg( \overline{\eta_j}^{\frac{1}{2}}) = \frac{\gamma}{2} \in [0,  \pi) $.
\end{thm}

\begin{proof}
The proof of this result involves specialising Theorem \ref{detthm}. With the same notations as Theorem \ref{detthm}, we set $v_i = \mathbf{cb}_i$, set $t_i = \eta_i$ for $i=1,2,3$, set $\lambda_4 =0$ and keep $ \mathbb{D} \ni \lambda_5 := \lambda \neq 0 $ arbitrary. Then for $i \neq j$, as $\textbf{cb}_i (\eta_j)$ is a constant multiplied by $ \langle \textbf{cb}_i, \textbf{cb}_j \rangle =0$, we must have $\textbf{cb}_i (\eta_j)=0$. So $[c_1 , c_2, c_3, c_4, c_5, \tilde{S}]$ simplifies to become a matrix of the form 
$$
\begin{pmatrix}
D & A_1 \\
0 & A_2
\end{pmatrix},
$$
where $D$ is a 3-by-3 diagonal matrix with each diagonal entry non-zero, 0 is the 3-by-3 zero matrix and $ A_1, A_2$ are $3$-by-$3$  matrices. So in this case $\det [c_1 , c_2, c_3, c_4, c_5, \tilde{S}] = \det D \det A_2$, and so $\det [c_1 , c_2, c_3, c_4, c_5, \tilde{S}] =0$ if and only if  $ \det A_2 = 0 $. We now argue $ \det A_2 = 0 $ if and only if \eqref{solntos6} holds.

If we write $\textbf{cb}_i = b_i k_{\eta_i}$ for $b_i \in \mathbb{C}$ (where $b_i$ is explicitly given in \eqref{explicitcb}), then  $\overline{\textbf{cb}_i (\lambda)} = \overline{b_i}\overline{ \langle k_{\eta_i}, k_{\lambda} \rangle } = \overline{b_i} k_{\lambda}(\eta_i)$, and so we can compute $A_2$ to be given by
$$
\begin{pmatrix}
x_4 & y_4 & s_4 \\
x_5 & y_5 & s_5 \\
x_6 & y_6 & s_6
\end{pmatrix}, $$
where 
\begin{align}
    x_4 = {\overline{b_2} {k_0 (\eta_2)}}{{\overline{b_1} {k_0 (\eta_1)}}}, \quad y_4 = {\overline{b_2} {k_\lambda (\eta_2)}}{{\overline{b_1} {k_\lambda (\eta_1)}}}, \\
    x_5 = {\overline{b_3} {k_0 (\eta_3)}{{\overline{b_1} {k_0 (\eta_1)}}}}, \quad y_5 = {\overline{b_3} {k_\lambda (\eta_3)}}{{\overline{b_1} {k_\lambda (\eta_1)}}} , \\
    x_6 = {\overline{b_3} {k_0 (\eta_3)}}{{\overline{b_2} {k_0 (\eta_2)}}}, \quad  y_6 = {\overline{b_3} {k_\lambda (\eta_3)}}{{\overline{b_2} {k_\lambda (\eta_2)}}} .
\end{align}
For ease of notation we define the constant $B_{\lambda} = 1 - \overline{\theta(\lambda)}\theta( \eta_1)$. Now noting that as $\theta( \eta_1) = \theta( \eta_2) = \theta( \eta_3)$, by \eqref{RK} we must have $k_0 (\eta_1 ) = k_0 (\eta_2 ) = k_0 (\eta_3 ) = B_0$, and similarly ${k_{\lambda} (\eta_i)}{k_{\lambda} (\eta_j)} = \frac{B_{{\lambda}}^2}{{(1-\overline{\lambda} \eta_i)(1-\overline{\lambda} \eta_j)}}$, and so from the above we obtain
\begin{align}\label{eqq}
    x_4 =& \overline{b_2 b_1} B_0^2, \quad y_4 = \overline{b_2 b_1}\frac{B_{{\lambda}}^2}{{(1-\overline{\lambda} \eta_2)(1-\overline{\lambda} \eta_1)}}, \\ \label{eqq3}
    x_5 =& \overline{b_3 b_1} B_0^2,  \quad y_5 = \overline{b_3 b_1} \frac{B_{{\lambda}}^2}{{(1-\overline{\lambda} \eta_3)(1-\overline{\lambda} \eta_1)}}, \\ \label{eqq2}
    x_6 =& \overline{b_3 b_2} B_0^2, \quad y_6 = \overline{b_3 b_2} \frac{B_{{\lambda}}^2}{{(1-\overline{\lambda} \eta_3)(1-\overline{\lambda} \eta_2)}}.
\end{align}

Expanding the determinant of $A_2 = \begin{pmatrix}
x_4 & y_4 & s_4 \\
x_5 & y_5 & s_5 \\
x_6 & y_6 & s_6
\end{pmatrix} $ via the third column we see  $\det A_2 = 0$ is equivalent to 
\begin{equation}\label{s6}
    s_6(x_4 y_5 - y_4 x_5)  = {s_4 (y_5 x_6 - x_5 y_6) + s_5 (x_4 y_6 -y_4 x_6)} .
\end{equation}
Using the values of $x_4, x_5, x_6 , y_4, y_5, y_6$ found in \eqref{eqq} \eqref{eqq3} \eqref{eqq2} we can write
\begin{align}
y_5 x_6 - x_5 y_6 &= \overline{b_3}^2 \overline{b_1 b_2} B_{\lambda}^2 B_0^2  \frac{1}{1 - \overline{\lambda} \eta_3} \left(\frac{1}{1 - \overline{\lambda} \eta_1} - \frac{1}{1 - \overline{\lambda} \eta_2} \right) \\
&= \overline{b_3}^2 \overline{b_1 b_2} B_{\lambda}^2 B_0^2  \frac{1}{1 - \overline{\lambda} \eta_3} \left(\frac{\overline{\lambda}( \eta_1 - \eta_2)}{(1 - \overline{\lambda} \eta_1)(1 - \overline{\lambda} \eta_2)}  \right), \\
x_4 y_6 - y_4 x_6 &= \overline{b_3 b_1} \overline{ b_2}^2 B_{\lambda}^2 B_0^2  \frac{1}{1 - \overline{\lambda} \eta_2} \left(\frac{1}{1 - \overline{\lambda} \eta_3} - \frac{1}{1 - \overline{\lambda} \eta_1} \right) \\
&= \overline{b_3 b_1} \overline{ b_2}^2 B_{\lambda}^2 B_0^2  \frac{1}{1 - \overline{\lambda} \eta_2} \left(\frac{\overline{\lambda}( \eta_3 - \eta_1)}{(1 - \overline{\lambda} \eta_3)(1 - \overline{\lambda} \eta_1)}  \right), \\
x_4 y_5 - y_4 x_5 &= \overline{b_3}\overline{ b_1}^2 \overline{ b_2} B_{\lambda}^2 B_0^2  \frac{1}{1 - \overline{\lambda} \eta_1} \left(\frac{1}{1 - \overline{\lambda} \eta_3} - \frac{1}{1 - \overline{\lambda} \eta_2} \right) \\
&= \overline{b_3}\overline{ b_1}^2 \overline{ b_2} B_{\lambda}^2 B_0^2 \frac{1}{1 - \overline{\lambda} \eta_1} \left(\frac{\overline{\lambda}( \eta_3 - \eta_2)}{(1 - \overline{\lambda} \eta_3)(1 - \overline{\lambda} \eta_2)}  \right).
\end{align}

At this stage we may now see that $x_4 y_5 - y_4 x_5 \neq 0$, since if $x_4 y_5 - y_4 x_5 =0$, this would imply $\eta_3 = \eta_2$, which can never be the case as pointed out in Section \ref{prelims}. Substituting the above values into \eqref{s6} gives
\begin{equation}\label{s6redo}
s_6 = \frac{1}{ \eta_3 - \eta_2} \left( s_4 \overline{\left(\frac{{b_3}}{{b_1}}\right)} ( \eta_1 - \eta_2 ) + s_5 \overline{\left(\frac{{b_2}}{{b_1}}\right)}(\eta_3 - \eta_1 ) \right).
\end{equation}

Finally as each $\left(\overline{\eta_{i}} \frac{\alpha+\theta(t)}{1+\overline{\theta(t)} \alpha}\right) \in \mathbb{T}$, for $i,j \in \{ 1,2,3 \}$, using the explicit description of $b_i$ given in \eqref{explicitcb}, and Lemma \ref{newlem},  we see $$\overline{\left(\frac{{b_i}}{{b_j}}\right)} = \frac{\left(\overline{\eta_{j}} \frac{\alpha+\theta(t)}{1+\overline{\theta(t)} \alpha}\right)^{\frac{1}{2}}\|k_{\eta_j}\|}{\left(\overline{\eta_{i}} \frac{\alpha+\theta(t)}{1+\overline{\theta(t)} \alpha}\right)^{\frac{1}{2}}\|k_{\eta_i}\|} = \frac{\overline{\eta_{j}}  ^{\frac{1}{2}}|\theta'(\eta_j)|^{\frac{1}{2}}}{\overline{\eta_{i}} ^{\frac{1}{2}}|\theta'(\eta_i)|^{\frac{1}{2}}},$$
where $\overline{\eta_{j}}  ^{\frac{1}{2}}$ is defined such that if $\arg( \overline{\eta_j}) := \gamma \in [0, 2 \pi)$, then $\arg( \overline{\eta_j}^{\frac{1}{2}}) = \frac{\gamma}{2}$. So \eqref{s6redo} simplifies to 
$$
s_6 = \frac{1}{ \eta_3 - \eta_2} \left( s_4 \left(\frac{\overline{\eta_{1}}  ^{\frac{1}{2}}|\theta'(\eta_1)|^{\frac{1}{2}}}{\overline{\eta_{3}} ^{\frac{1}{2}}|\theta'(\eta_3)|^{\frac{1}{2}}}\right) ( \eta_1 - \eta_2 ) + s_5 \left(\frac{\overline{\eta_{1}}  ^{\frac{1}{2}}|\theta'(\eta_1)|^{\frac{1}{2}}}{\overline{\eta_{2}} ^{\frac{1}{2}}|\theta'(\eta_2)|^{\frac{1}{2}}}\right)(\eta_3 - \eta_1 ) \right). \qedhere
$$ 

\end{proof}

\begin{rem}The above theorem may be viewed as a 3-by-3 version of a generalisation of Theorem 1.15 in \cite{TTOonfinitedimn}. In \cite{TTOonfinitedimn} the authors only consider the case when $t=0$ in \eqref{clarkoperator}, whereas our result holds for all $t \in \mathbb{D}$. Furthermore we note that our condition for $S$ to be a matrix representation of a TTO with respect to $\textbf{cb}_i$ eliminates all variables apart from $\eta_i$ for $i = 1,2,3$.
\end{rem}

One can make a corollary to the above theorem which provides numerous examples of matrices which show Question \ref{gotoQ1} to be negative.

\begin{cor}
Consider a symmetric matrix, $S$, of the form
\begin{equation}\label{moreexmaples}
S=    \begin{pmatrix}
a & 0 & 1 \\
0 & b & 0 \\
1 & 0 & c
\end{pmatrix}, \quad S= \begin{pmatrix}
a & 1 & 0 \\
1 & b & 0 \\
0 & 0 & c
\end{pmatrix} \text{  or } S = \begin{pmatrix}
a & 0 & 0 \\
0 & b & 1 \\
0 & 1 & c
\end{pmatrix} 
\end{equation}
where $a, b , c \in \mathbb{R}$. Then $S$ is unitarily equivalent to a TTO, but is not the matrix representation of a TTO with respect to a modified Clark basis.
\end{cor}
\begin{proof}
In every case we have that $S$ is a normal matrix and so is unitarily equivalent to a TTO by Theorem 5.4 in \cite{TTOUEsimilarity}. Clearly as $b_i \neq 0$, this means $\overline{\frac{b_i}{b_j}} = \left(\frac{\overline{\eta_{j}}  ^{\frac{1}{2}}|\theta'(\eta_j)|^{\frac{1}{2}}}{\overline{\eta_{i}} ^{\frac{1}{2}}|\theta'(\eta_i)|^{\frac{1}{2}}}\right) \neq 0$ for $i, j  = 1,2,3$. Thus by the previous theorem $S$ does not satisfy \eqref{solntos6} for any choice of distinct $\eta_1, \eta_2, \eta_3 \in \mathbb{T}$, and so $S$ can not be the matrix representation of a TTO with respect to a modified Clark basis.
\end{proof}

Naturally, the above working leads us to consider another conjecture which is a modification of Question \ref{gotoQ1}.

\begin{conj}\label{newconj}
Suppose that $M$ is a complex symmetric matrix. If $M$ is unitarily equivalent to a TTO, does there exist an inner function $u$, and a $C_{\theta}$-real basis for the model space ${K}_{u}$ such that $M$ is the matrix representation of a TTO on $K_u$ with respect to this basis? In other words, do all such unitary equivalences between complex symmetric matrices and TTOs arise from $C_{\theta}$-real matrix representations?
\end{conj}

We observe that if the above conjecture is true then by Proposition \ref{generalCreal} (see below) every complex symmetric matrix $M$ which is unitarily equivalent to a TTO is orthogonally equivalent to the matrix representation of a TTO with respect to a modified Clark basis. 

Finally we remark that, as eluded to in \cite{recentprogressGarcia}, there is also a possibility that the matrix-valued truncated Toeplitz operator may play a role in the modelling of complex symmetric operators. Matrix-valued truncated Toeplitz operators are a vector-valued generalisation of the truncated Toeplitz operator and have been studied in \cite{MTTOpaper, khan2018matrix, khan2020generalized}.

\section{A geometric approach}\label{sec4}
Although function theory was used to construct one class of $C_{\theta}$-real bases, the modified Clark bases, it turns out that when we have the description of one $C_{\theta}$-real basis for a model space $K_{\theta}$, we can use a purely algebraic method to generate all $C_{\theta}$-real bases for $K_{\theta}$. Furthermore, one can describe the matrix representation of a TTO with respect to given $C_{\theta}$-real basis by relating this to the matrix representation of the TTO with respect to a modified Clark basis. We show this with the following proposition.

\begin{rem}
The term orthogonal is typically used for a real matrix $A$ such that $A A^T = A^T A = I_d$, however we note that orthogonal matrices are real unitary matrices. 
\end{rem}

\begin{prop}\label{generalCreal}
Let $\theta$ be of order $n$. Given a $C_{\theta}$-real basis for $K_{\theta}$, $v_i$, and a modified Clark basis for $K_{\theta}$, $\mathbf{cb}_i$, there exists an orthogonal matrix $U_r$ such that $v_i = T U_r e_i$, where $T: \mathbb{C}^n \to K_{\theta}$, $e_i \mapsto \mathbf{cb}_i$. Conversely, given any orthogonal $U_r$ and any modified Clark basis $\mathbf{cb}_i$, $T U_r e_i$ is a $C_{\theta}$-real basis for $K_{\theta}$. 

Furthermore for $v_i = T U_r e_i$ and $A \in \mathcal{T}_{\theta}$ we have 
$$
[A]_{v_i} = U_r^{-1}[A]_{\mathbf{cb}_i} U_r.
$$
\end{prop}

\begin{proof}
Given a $C_{\theta}$-real basis $v_i$, and a modified Clark basis $\mathbf{cb}_i$ there clearly exists constants $r_{ij}$ such that $v_i = r_{i1}\mathbf{cb}_1 + ... + r_{in}\mathbf{cb}_n$. As $v_i$ and $\mathbf{cb}_i$ are $C_{\theta}$-invariant, we must have 
$$
r_{i1}\mathbf{cb}_1 + ... + r_{in}\mathbf{cb}_n = v_i = C_{\theta} (v_i) = \overline{r_{i1}} \mathbf{cb}_1 + ... + \overline{r_{in}}\mathbf{cb}_n ,
$$
and so for every $i, \, j \in \{ 1, ... , n \}$ we must have $r_{ij} = \overline{r_{ij}} \in \mathbb{R}$. This establishes that $v_i = T U_r e_i$, where $U_r$ is a real matrix. To show $U_r$ is unitary we observe that $T U_r$ is unitary, as it sends an orthonormal basis to an orthonormal basis, and so is $T$, so $U_r$ must also be unitary. 

To show the converse part of the proposition one can readily see that each $T U_r e_i$, being a real linear combination of a $C_{\theta}$-real basis, is itself $C_{\theta}$-invariant, and $T U_r e_1, ..., T U_r e_n$ is an orthonormal basis because $ e_1, ..., e_n$ is an orthonormal basis and $T U_r$ is unitary. In order to prove the second statement of the proposition, observe that 
$$
[A]_{v_i} = U_r^{-1} T^{-1} A T U_r = U_r^{-1} [A]_{\mathbf{cb}_i} U_r . \qedhere
$$
\end{proof}
\begin{rem}
By the above proposition we may observe that if a 3-by-3 matrix $M$ is such that there exist an orthogonal matrix $U_r$ with $U_r M U_r^{-1} = D$ where $D$ is a diagonal matrix with entries $d_1, d_2, d_3$, then $M$ is a representation of a TTO with respect to a $C_{\theta}$-real basis. Indeed, if we consider $A:= \sum_{i=1}^3 d_i \mathbf{cb}_i \otimes \mathbf{cb}_i$, for some modified Clark basis $\mathbf{cb}_i$ then $[A]_{\mathbf{cb}_i} = D$, and so for the basis $v_i = T U_r e_i$ we have $[A]_{v_i} = U_r^{-1} [A]_{\mathbf{cb}_i} U_r = U_r^{-1} D U_r = M$. With this observation we see that although the matrices in \eqref{moreexmaples} can not be a modified Clark basis representation of a TTO, they are a matrix representation of a TTO with respect to some $C_{\theta}$-real basis.
\end{rem}

In the remainder of this section we will exploit the above proposition and show that the problem of determining whether a given 3-by-3 symmetric matrix, $M$, is a matrix representation of a TTO on some fixed model space with respect to a $C_{\theta}$-real basis can actually be rephrased purely as a geometric problem.

By Proposition \ref{generalCreal} we can see that a given symmetric matrix
$$
M= \begin{pmatrix}
m_1 & m_4 & m_5 \\
m_4 & m_2 & m_6 \\
m_5 & m_6 & m_3 
\end{pmatrix},
$$ is a matrix representation of a TTO on $K_{\theta}$ with respect to a $C_{\theta}$-real basis if and only if for a modified Clark basis for $K_{\theta}$, $\mathbf{cb}_i$, there exists a orthogonal matrix $U_r$ and a $A \in \mathcal{T}_{\theta}$ such that
\begin{equation}\label{observation}
    U_r M U_r^{T}=  [A]_{\mathbf{cb}_i} .
\end{equation}
In fact, this is actually equivalent to the existence of a special orthogonal matrix $U_s$ such that $U_s M U_s^{T}=  [A]_{\mathbf{cb}_i}$, as if \eqref{observation} holds for some orthogonal $U_r$ with $\det U_r = -1$ then setting $U_s = - U_r$ we see $U_s M U_s^{T}=  [A]_{\mathbf{cb}_i}$.

Given matrices $M$ and $S$, when trying to solve a matrix equation of the form $XMX^T = S$ for $X$, we may actually rephrase this equation as a problem in algebraic geometry. In the case of \eqref{observation} we have the following.

\begin{thm}
Let $\theta$ be a Blaschke product of order 3, let $\mathbf{cb}_i$ be a modified Clark basis for $K_{\theta}$ with corresponding points on the boundary $\eta_i$ given by \eqref{above}. Consider the system of polynomials with variables $r_1, r_2, ... ,r_9$ defined by
\begin{equation}\label{R1}
r_1^2 + r_4^2 + r_7^2 =1, \quad r_2^2 + r_5^2 + r_8^2 =1, \quad r_3^2 + r_6^2 + r_9^2 =1,
\end{equation}
\begin{equation}\label{R2}
r_1 r_2 + r_4 r_5 + r_7 r_8 =0, \quad r_2 r_3 + r_5 r_6 + r_8 r_9 =0, \quad r_1 r_3 + r_4 r_6 + r_7 r_9 =0,
\end{equation}
\begin{equation}\label{R3}
    (\eta_3 - \eta_2) A_6 - \left(\frac{\overline{\eta_{1}}  ^{\frac{1}{2}}|\theta'(\eta_1)|^{\frac{1}{2}}}{\overline{\eta_{3}} ^{\frac{1}{2}}|\theta'(\eta_3)|^{\frac{1}{2}}}\right) ( \eta_1 - \eta_2 ) A_4 - \left(\frac{\overline{\eta_{1}}  ^{\frac{1}{2}}|\theta'(\eta_1)|^{\frac{1}{2}}}{\overline{\eta_{2}} ^{\frac{1}{2}}|\theta'(\eta_2)|^{\frac{1}{2}}}\right)(\eta_3 - \eta_1 ) A_5 =0,
\end{equation}
where the square root is given by the same convention as in \eqref{solntos6} and where 
\newline $A_6 = m_1  r_4  r_7+m_4  r_5  r_7+m_5  r_6  r_7+m_4  r_4  r_8+m_2  r_5  r_8+m_6  r_6  r_8+m_5  r_4  r_9+m_6  r_5  r_9+m_3  r_6  r_9$,
\vskip 0.1cm
\noindent $A_4 =m_1r_1 r_4+m_4 r_2 r_4+m_5 r_3 r_4+m_4 r_1 r_5+m_2 r_2 r_5+m_6 r_3 r_5+m_5 r_1 r_6+m_6 r_2 r_6+m_3 r_3 r_6$
\vskip 0.1cm
\noindent $A_5 = m_1  r_1  r_7+m_4  r_2  r_7+m_5  r_3  r_7+m_4  r_1  r_8+m_2  r_2  r_8+m_6  r_3  r_8+m_5  r_1  r_9+m_6  r_2  r_9+m_3  r_3  r_9 $.

\noindent Then 
$$
M= \begin{pmatrix}
m_1 & m_4 & m_5 \\
m_4 & m_2 & m_6 \\
m_5 & m_6 & m_3 
\end{pmatrix}
$$
is a matrix representation of a TTO on $K_{\theta}$ with respect to a $C_{\theta}$-real basis if and only if equations \eqref{R1} \eqref{R2} and \eqref{R3} are simultaneously satisfied with a real solution.
\end{thm}

\begin{proof}
As shown by \eqref{observation}, we see that $M$ is a matrix representation of a TTO on $K_{\theta}$ with respect to a $C_{\theta}$-real basis if and only if there exists orthogonal matrix $\left(\begin{array}{lll}
r_{1} & r_{2} & r_{3} \\
r_{4} & r_{5} & r_{6} \\
r_{7} & r_{8} & r_{9}
\end{array}\right)$ and a $A \in \mathcal{T}_{\theta}$ such that 
\begin{equation}\label{realalggeom}
    \left(\begin{array}{lll}
r_{1} & r_{2} & r_{3} \\
r_{4} & r_{5} & r_{6} \\
r_{7} & r_{8} & r_{9}
\end{array}\right) M \left(\begin{array}{lll}
r_{1} & r_{2} & r_{3} \\
r_{4} & r_{5} & r_{6} \\
r_{7} & r_{8} & r_{9}
\end{array}\right)^{T} = [A]_{\mathbf{cb}_i}.
\end{equation}
Now equations \eqref{R1} \eqref{R2} being satisfied (with a real solution) are equivalent to $$\left(\begin{array}{lll}
r_{1} & r_{2} & r_{3} \\
r_{4} & r_{5} & r_{6} \\
r_{7} & r_{8} & r_{9}
\end{array}\right)$$ being a orthogonal matrix. Furthermore, by multiplying out the left hand side of \eqref{realalggeom}, we see that the left hand side of \eqref{realalggeom} is a representation of a TTO on $K_{\theta}$ with respect to $\mathbf{cb}_i$ (i.e. that the condition from Theorem \ref{clarkbasisrepn} is satisfied) if and only if \eqref{R3} is satisfied.
\end{proof}

\begin{rem}
We notice that in the above theorem the condition that $M$ is a matrix representation of a TTO on $K_{\theta}$ with respect to a $C_{\theta}$-real basis is independent of our original choice of $\mathbf{cb}_i$.
\end{rem}

\begin{rem}
In the above theorem, there is a procedure to find a choice of $\eta_1, \eta_2, \eta_3$, which is the following. Consider the modified Clark basis for $K_{\theta}$ with parameters in \eqref{above} \eqref{explicitcb} given by $\alpha=1$ and $t$ such that $\theta(t) = 0$. Then in order to find $\eta_i$ one must solve $\theta(z) = \prod_{i=1}^3 \frac{z-\lambda_i}{1- \overline{\lambda_i}z} =1$, or equivalently, after multiplying both sides by $\prod_{i=1}^3 1- \overline{\lambda_i}z$ and rearranging
$$
z^3 K_3 - z^2 K_2 + z K_1 - K_0 = 0,
$$
where $K_3 = 1 + \overline{ \lambda_1 \lambda_2 \lambda_3} , \, K_2=  \lambda_1 + \lambda_2 + \lambda_3 + \overline{\lambda_1 \lambda_2} + \overline{\lambda_1 \lambda_3} + \overline{\lambda_2 \lambda_3}, \, K_1 = \lambda_1 \lambda_2 + \lambda_1 \lambda_3 + \lambda_2 \lambda_3 + \overline{\lambda_1}+ \overline{\lambda_2}+ \overline{\lambda_3}, \, K_0 = \lambda_1 \lambda_2 \lambda_3 +1$.
\end{rem}

\section*{Acknowledgements}
The author is grateful to EPSRC for financial support.
\newline
The author is grateful to Dr Ben Sharp for his advice and guidance during the preparation of this paper.
\section*{Statements and Declarations}
This study was funded by the EPSRC.
\newline 
Declarations of interest: none.

\newpage 

\bibliographystyle{plain}
\bibliography{bibliography.bib}

\begin{thebibliography}{10}

\bibitem{baranov2010symbols}
A.~Baranov, R.~Bessonov, and V.~Kapustin.
\newblock Symbols of truncated {T}oeplitz operators.
\newblock {\em J. Funct. Anal.}, 261(12):3437--3456, 2011.

\bibitem{baranov2010bounded}
A.~Baranov, I.~Chalendar, E.~Fricain, J.~Mashreghi, and D.~Timotin.
\newblock Bounded symbols and reproducing kernel thesis for truncated
  {T}oeplitz operators.
\newblock {\em Journal of Functional Analysis}, 259(10):2673--2701, 2010.

\bibitem{nonhermitianphysics}
C.~M. Bender.
\newblock Introduction to {P}{T}-symmetric quantum theory.
\newblock {\em Contemp.Phys.}, 46(4):277--292, 2005.

\bibitem{bessenovfredholm}
R.~V. Bessonov.
\newblock Fredholmness and compactness of truncated {T}oeplitz and {H}ankel
  operators.
\newblock {\em Integral Equations Operator Theory}, 82(4):451--467, 2015.

\bibitem{MR3398735}
M.~C. C\^amara and J.~R. Partington.
\newblock Spectral properties of truncated {T}oeplitz operators by equivalence
  after extension.
\newblock {\em J. Math. Anal. Appl.}, 433(2):762--784, 2016.

\bibitem{TTOUEsimilarity}
J.~A. Cima, S.~R. Garcia, W.~T. Ross, and W.~R. Wogen.
\newblock Truncated {T}oeplitz operators: spatial isomorphism, unitary
  equivalence, and similarity.
\newblock {\em Indiana Univ. Math. J.}, 59(2):595--620, 2010.

\bibitem{cima2000backward}
J.~A. Cima and W.~T. Ross.
\newblock {\em The backward shift on the Hardy space}.
\newblock American Mathematical Soc., 2000.

\bibitem{TTOonfinitedimn}
J.~A. Cima, W.~T. Ross, and W.~R. Wogen.
\newblock Truncated {T}oeplitz operators on finite dimensional spaces.
\newblock {\em Oper. Matrices}, 2(3):357--369, 2008.

\bibitem{duren1970theory}
P.~L. Duren.
\newblock Theory of ${H}^p$ spaces.
\newblock {\em Pure Appl. Math}, 38:74, 1970.

\bibitem{CCOgarcia}
S.~R. Garcia.
\newblock Conjugation and {C}lark operators.
\newblock In {\em Recent advances in operator-related function theory}, volume
  393 of {\em Contemp. Math.}, pages 67--111. Amer. Math. Soc., Providence, RI,
  2006.

\bibitem{garciaopen}
S.~R. Garcia.
\newblock Three questions about complex symmetric operators.
\newblock {\em Integral Equations Operator Theory}, 72(1):3--4, 2012.

\bibitem{modelspacesgarcia}
S.~R. Garcia, J.~Mashreghi, and W.~T. Ross.
\newblock {\em Introduction to model spaces and their operators}, volume 148 of
  {\em Cambridge Studies in Advanced Mathematics}.
\newblock Cambridge University Press, Cambridge, 2016.

\bibitem{FBPandconnections}
S.~R. Garcia, J.~Mashreghi, and W.~T. Ross.
\newblock {\em Finite {B}laschke products and their connections}.
\newblock Springer, Cham, 2018.

\bibitem{CSO2}
S.~R. Garcia and D.~E. Poore.
\newblock On the closure of the complex symmetric operators: compact operators
  and weighted shifts.
\newblock {\em J. Funct. Anal.}, 264(3):691--712, 2013.

\bibitem{garciaUETTOanalyticsymbols}
S.~R. Garcia, D.~E. Poore, and W.~T. Ross.
\newblock Unitary equivalence to a truncated {T}oeplitz operator: analytic
  symbols.
\newblock {\em Proc. Amer. Math. Soc.}, 140(4):1281--1295, 2012.

\bibitem{garciaapp}
S.~R. Garcia, E.~Prodan, and M.~Putinar.
\newblock Mathematical and physical aspects of complex symmetric operators.
\newblock {\em J. Phys. A}, 47(35):353001, 54, 2014.

\bibitem{recentprogressGarcia}
S.~R. Garcia and W.~T. Ross.
\newblock Recent progress on truncated {T}oeplitz operators.
\newblock {\em Blaschke Products and Their Applications}, 65:265--319, 2013.

\bibitem{CSO1}
S.~Jung, Y.~Kim, E.~Ko, and J.~E. Lee.
\newblock Complex symmetric weighted composition operators on {$H^2(\Bbb{D})$}.
\newblock {\em J. Funct. Anal.}, 267(2):323--351, 2014.

\bibitem{khan2020generalized}
R.~Khan.
\newblock The generalized {C}rofoot transform.
\newblock {\em Oper. Matrices}, 15(1):225--237, 2021.

\bibitem{khan2018matrix}
R.~Khan and D.~Timotin.
\newblock Matrix valued truncated {T}oeplitz operators: basic properties.
\newblock {\em Complex Analysis and Operator Theory}, 12(4):997--1014, 2018.

\bibitem{nikolski2002operators}
N.~K. Nikolski.
\newblock {\em Operators, Functions, and Systems-An Easy Reading: Hardy,
  Hankel, and Toeplitz}, volume~1.
\newblock American Mathematical Soc., 2002.

\bibitem{CSO3}
S.~W. Noor.
\newblock On an example of a complex symmetric composition operator on
  {$H^2(\Bbb{D})$}.
\newblock {\em J. Funct. Anal.}, 269(6):1899--1901, 2015.

\bibitem{oloughlin2020nearly}
R.~O'Loughlin.
\newblock Nearly invariant subspaces with applications to truncated {T}oeplitz
  operators.
\newblock {\em Complex Anal. Oper. Theory}, 14, 2020, no.8, 86.

\bibitem{mythesis}
R.~O'Loughlin.
\newblock Multidimensional {T}oeplitz and truncated {T}oeplitz operators.
\newblock \url{https://etheses.whiterose.ac.uk/29284/}, University of {L}eeds,
  2021.

\bibitem{MTTOpaper}
R.~O'Loughlin.
\newblock Matrix-valued truncated {T}oeplitz operators: unbounded symbols,
  kernels and equivalence after extension.
\newblock {\em Integral Equations Operator Theory}, 94(1):Paper No. 5, 20,
  2022.

\bibitem{symbolsofcompact}
R.~O'Loughlin.
\newblock Symbols of compact truncated {T}oeplitz operators.
\newblock {\em Journal of Mathematical Analysis and Applications},
  507(2):125819, 2022.

\bibitem{sarason2007algebraic}
D.~Sarason.
\newblock Algebraic properties of truncated {T}oeplitz operators.
\newblock {\em Oper. Matrices}, 1(4):491--526, 2007.

\bibitem{UETTOstrouse}
E.~Strouse, D.~Timotin, and M.~Zarrabi.
\newblock Unitary equivalence to truncated {T}oeplitz operators.
\newblock {\em Indiana Univ. Math. J.}, 61(2):525--538, 2012.

\end{thebibliography}

\end{document}